\pdfoutput=1 
\documentclass[12pt]{article}
\usepackage{amsmath}

\usepackage{amssymb}
\usepackage{amsthm}
\usepackage{alltt}
\usepackage{amscd}
\usepackage[mathscr]{eucal}
\usepackage{mathrsfs}
\usepackage[all]{xy}
\usepackage{verbatim}
\usepackage{moreverb}
\usepackage{graphicx, transparent, color}
\DeclareMathAlphabet{\mathpzc}{OT1}{pzc}{m}{it}
\theoremstyle{plain}

\newtheorem{thm}{Theorem}[section]

\newtheorem{cor}{Corollary}
\newtheorem{lemma}[thm]{Lemma}
\newtheorem{prop}[thm]{Proposition}
\theoremstyle{definition}
\begin{document}
{\center\bf\large
Linear Configurations of Complete Graphs $K_4$ and $K_5$ in $\mathbb R^3$, and Higher Dimensional Analogs \\
\small  Andrew Marshall\\ 
\today \vspace{1cm}\\}

 \begin{abstract}

We investigate the space $C(X)$ of images of linearly embedded skeleta of simplices $X$ in $\mathbb R^n$, for two families of codimension 2 complexes, each ranging over $n$.  In the first family, $X=K$ is the $(n-2)$-skeleton of the $n$-simplex.  In the second family, $X=L$ is the $(n-2)$-skeleton of the $(n+1)$-simplex.  The main result is that for $n>2$, $C(X)$ (for either $X=K,L$) deformation retracts to a subspace homeomorphic to the double mapping cylinder \[SO(n)/A_{n+1}\leftarrow SO(n)/A_n\rightarrow SO(n)/S_n,\] where $A_n$ is the alternating group and $S_n$ the symmetric group.  The resulting fundamental group provides an example of a generalization of the braid group, which is the fundamental group of a configuration of points in the plane.  This group is presented, for the case $n=3$, and its action on $F_3$ is presented.

\end{abstract}

\section{Introduction} The braid group on $n$ strands is the fundamental group of the configuration space of $n$ points in the plane.   This group has enjoyed prominence throughout many areas of mathematics and mathematical physics including group theory; the $n$-body problem and symplectic geometry; cryptology; robotic control; knot theory.   When considered as the fundamental group of a space of embeddings mod parametrizations $E(X,Y)/Aut(X)$, (in this case, for $X$ the set of $n$ points, $Y=\mathbb R^2$), a natural question is: what spaces arise from configurations of $X$ in $Y$ for other arguments $X,Y$?  This question has been studied where $Y$ is an assortment of other spaces including surfaces, graphs and lens spaces (see \cite{sbraids4},\cite{sbraids3},\cite{sbraids1}; \cite{ongraphs1},\cite{ongraphs2}; \cite{invariant}), and also where $X$ is more than a discrete space.  In \cite{Dahm} the {\em symmetric automorphism group} is introduced as the fundamental group of the configuration space of $n$ disjoint, unlinked (unknotted) $C^\infty$ embedded circles in $\mathbb R^3$.  In \cite{wickets} the same space was shown to deformation retract to the space of $n$ unlinked circles in $\mathbb R^3$.  In this paper we consider the configuration spaces $C(K_4),C(K_5)$ of linearly embedded complete graphs $K_4,K_5$ in $\mathbb R^3$ and their analogs in higher dimensions: the $(n-2)$-skeleton of the $n$-simplex, denoted by $K$, and the $(n-2)$-skeleton of the $(n+1)$-simplex, denoted by $L$, where each are linearly embedded in $\mathbb R^n$.  The main result is that for $n>2$, the configuration spaces $C(K)$ and $C(L)$ each deformation retract to a subspace homeomorphic to the double mapping cylinder 
\[SO(n)/A_{n+1}\leftarrow SO(n)/A_n\rightarrow SO(n)/S_n,\]
and are therefore homotopy equivalent.  It is noted that this homotopy equivalence does not hold when we increase the number of vertices to get $C(K_6)$, and higher dimensional analogs.  In \cite{Huh} it is shown that a linearly embedded $K_6$ in $\mathbb R^3$ can have either one or three Hopf links, and so in particular $C(K_6)$ is not connected.

In Section 2 we define a lowest dimensional compact subspace called pyramids $P(X)$ of the configuration spaces $C(X)$, $X=C,L$ and show that the respective configuration spaces deformation retract to the pyramidal spaces, and that $P(K)$ is homeomorphic to $P(L)$.  The main result is this stated in terms of the actual homotopy type, (i.e., the double mapping cylinder of the previous paragraph).  The main tool used is an $O(n)$-equivariant Gram Schmidt process, which is conjugated to produce a deformation retraction from the configuration space of simplices to the subspace of regular simplices.  We make use of a combinatorial result of Radon's to limit the types of degeneracies that can occur in $C(K)$ and $C(L)$.  We conclude this section with corollary results about the spaces of embeddings $E(X)=Emb_{Linear}(X,\mathbb R^n)$, $X=L,K$ which cover the corresponding configuration spaces.

Section 3 contains two alternative methods for regularization, presented for their geometric appeal, as well as a recipe for deformation retracting $C(K)$ to $P(K)$ for a generic regularization $\mathscr R$.

Section 4 gives two presentations of the fundamental group of our space of interest and describes the action of this group on $F_n$.

I would like to thank Jason Anema, Lucien Clavier, Charles Marshall, and Jimmy Mathews for helpful discussions on the development of these ideas.  I am especially indebted to my adviser Allen Hatcher, for ideas and guidance in my doctoral program while this material was developed.

%
%
%
%  DEFORMATION RETRACTIONS TO PYRAMIDAL SPACE
%
%
%

\section{Deformation Retractions to Pyramidal Space}
\label{section 1}

\begin{figure} \centering \includegraphics[width=2.5in]{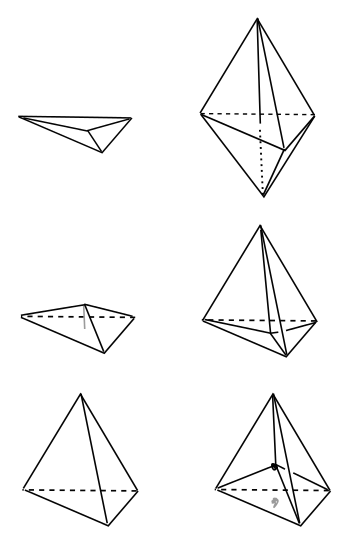} \caption{A comparison of the symmetries found in the pyramidal cases for $n=3$.  Stabilizers in $SO(3)$ for these configurations are, from top to bottom, $S_3$, $A_3$, and $A_4$, which give dihedral, rotational, and tetrahedral symmetries, respectively.}
\label{7.12} \end{figure} 

As defined in the introduction, $K$ is the codimension 2 skeleton of the $n$-simplex, $L$ is the codimension 3 skeleton of the $(n+1)$-simplex, and $C(\cdot)$ with either argument is the configuration space of the argument in $\mathbb R^n$.   

In each of $C(K)$ and $C(L)$ there is a subspace of configurations which enjoy $S_n$ symmetry.  In $C(K)$ these consist of complexes such that $n$ of the points are vertices of a regular $(n-1)$-simplex, while the other vertex lies on the line perpendicular to this simplex and through its barycenter.  In $C(L)$, some $n+1$ vertices are in the same position just described for $C(K)$, but the final vertex lies on the same line, so that two vertices are on a line which is perpendicular to the $(n-1)$-simplex spanning the other $n$ vertices and which passes through the barycenter of this simplex.  Let $P(K)$ be the aforementioned subset of $C(K)$ but where we fix the edges of the regular face to be unit length, fix the barycenter to be at the origin, and truncate the height (i.e., the distance of the vertex in the non-symmetric direction from the barycenter of its opposite face) to be between 0 and that of a regular unit-edged $n$-simplex.  Similarly, we let $P(L)$ be those $S_n$-symmetric configurations in $C(L)$ whose vertices form one regular $n$-simplex $\Delta$, and also form one $P(K)$ configuration $\langle \! |$ sharing a face $F$ of $\Delta$ such that if the apex $a$ of $\langle \! |$ is contained in $\Delta$, then $a$ is between the barycenter of $\Delta$ and the barycenter of $F$.  Here too, we fix the edges of $F$ to be unit length and put the barycenter of the $(n+1)$-simplex at the origin.  In both cases, we call such configurations {\em pyramids}.

\begin{prop}\label{cylinderprop}
Both pyramidal spaces $P(K)$ and $P(L)$ are homeomorphic to the double mapping cylinder $SO(n)/A_{n+1}\leftarrow SO(n)/A_n\rightarrow SO(n)/S_n$, where the maps are dual to inclusions $A_{n+1}\leftarrow A_n\rightarrow S_n.$
\end{prop}
\begin{proof} This is more or less evident from the description of the pyramidal spaces (see Figure \ref{7.12}).  Both $P(K)$, $P(L)$ decompose into a line segment's worth of configurations which have $A_n$ for a stabilizer in $SO(n)$.  One end of this interval is glued to the configurations with stabilizer $S_n\subset SO(n)$ (in $P(K)$ these are degenerate, being contained in a hyperplane) via the double cover induced from the inclusion $A_n<S_n$.  The other end is glued to the space of configurations with stabilizer $A_n\subset SO(n)$ (in $P(K)$ these are regular simplices) via the $(n+1)$-fold cover induced from an inclusion $A_n<A_{n+1}$.
\end{proof}

The main result then follows if we show the existence of deformation retractions from each configuration space to their associated pyramidal spaces.  The general strategy in the $C(K)$ case will be to use two different deformation retractions: one, a vertex-label-invariant regularization for simplices which are far from degenerate, and two, a vertex specific deformation retraction for those which are near (or in fact) degenerate.  The subtlety here will be in showing the two glue together continuously.  The general strategy for $C(L)$ will be similar, with only minor adjustments. 

We will make use of Radon's theorem twice, which says
\begin{thm}\label{Radons} Any $n+2$ points in $\mathbb R^n$ can be partitioned into two subsets $U_1$, $U_2$ so that ${\rm convex~hull}(U_1)\cap{\rm convex~hull}(U_2)\ne \varnothing.$
\end{thm}

For a proof see \cite{Ziegler}.  Points in the non-empty intersection are called {\em Radon points}.  The first application is to understand the degeneracies in $C(K)$.  Either $x\in C(K)$ spans a non-degenerate simplex or $x$ is contained in a codimension 1 hyperplane.  A degeneracy of greater codimension is not possible, since any $n$ vertices in $x$ span a non-degenerate $(n-1)$-simplex, as they belong to the simplicial sphere that is the $(n-2)$-skeleton of the $(n-1)$-simplex.  By Radon's theorem, if $x$ is contained in a codimension 1 hyperplane, it cannot be that all $n+1$ vertices are extremal, since in this case the least numerous of $U_i$ must contain at least 2 vertices, and hence an edge must intersect its opposite face, both of which belong to $x$.  Therefore, for degeneracies in $C(K)$, exactly 1 vertex is in the convex hull of the others.

A {\em solid angle} of a solid cone in $\mathbb R^n$ (for our purposes, a cone will be the convex hull of $n$ rays from a common {\em cone point}) is defined to be the $(n-1)$ dimensional volume of the intersection of the cone with a unit sphere centered at the cone point.  A solid angle of a vertex $v$ of a simplex is the solid angle of the cone formed by the edges incident to $v$.  

\begin{lemma} \label{solidangles} The sum of the solid angles of an $n$-simplex $x$ in $\mathbb R^n$ is bounded from above by half the volume of the unit $(n-1)$-sphere, and below by 0, for $n>1$.  These are tight bounds.
\end{lemma}

\begin{proof}

For each of the $n+1$ vertices in $x$, translate a copy of $x$ so that its $i$th vertex is at $0$.  Let $C_i$ denote the interior of the $i$th cone formed by extending each incident edge outwards, and denote with $-C_i$ its reflection through $0$.  In $C_0$ we have positive coordinates $(a_1,\ldots,a_n)$ representing coefficients of the $n$ vertices of $x$ excluding the origin.  Then \[C_0={\rm positive~span }\{x_i\}=\sum a_i x_i,\] for 
$a_i\ge 0$, and 
\begin{align*}C_i&={\rm positive~span}(\{x_j-x_i\}_{j\ne i}\cup\{-x_i\})\\
&=(\sum_j a_j(x_j-x_i))-a_ix_i.\end{align*}
Putting these into coordinates in $x_i$ gives $C_i$ as 
\[(a_1,a_2,\ldots,a_{i-1},-\sum_j a_j,a_{i+1},\ldots, a_n),\]
from which it is clear that for any $i\ne j$ we have 
\[int(C_i)\cap int(C_j)=int(C_i)\cap int(-C_j)=\varnothing.\]

Intersection with a unit sphere $S^{n-1}$ then gives that 
\[Vol(S^{n-1})\ge \sum_i~Vol(C_i\cap S^{n-1})+\sum_i~Vol(-C_i\cap S^{n-1})=2S,\]
for $S$ the sum of the solid angles. (This argument only fails in the case $n=1$. For $n=2$ the inequality is an equality.)  The bounds are tight since a near degenerate simplex can be made to have one solid angle which approaches a hemisphere (here one vertex is close to being in the convex hull of the others) or made so that each solid angle is arbitrarily close to 0 (here one edge is close to intersecting its opposite $(n-2)$-face).  We extend the greatest solid angle to be $\frac 1 2 Vol(S^{n-1})$ on those degenerate configurations with one vertex in the convex hull of the others.
\end{proof}

The effect of this lemma is that we have a surjective function \[\alpha:C(K)\rightarrow (0,V],\] where $V=\frac 1 2 Vol(S^{n-1})$, which gives the greatest solid angle, and for $x\in (\frac 1 2 V, V]$ there is only the one vertex with a solid angle in this range.

We will provide a way to regularize simplices far away from $\alpha^{-1}(V)$ by first giving an orthogonalization which is $O(n)$-equivariant (in particular, equivariant to vertex labeling). 

\begin{lemma} The linear deformation retraction 
\[\Phi_t(x)=(1-t)x+tx(x^\intercal x)^{-1/2}\] is equivariant under the right action of $O(n)$ and terminates in $O(n)$. 
\end{lemma}
\begin{proof} First, the inverse of the square root is defined for $x^\intercal x$, as this is a positive definite matrix, and so is diagonalizable with a diagonal of positive eigenvalues.  Next, $x(x^\intercal x)^{-1/2}\in O(n)$ since 
\begin{align*}
x(x^\intercal x)^{-1/2}[x(x^\intercal x)^{-1/2}]^\intercal & = x(x^\intercal x)^{-1} x^\intercal \\&= I.
\end{align*}
Finally, for $Q\in O(n),$ 
\begin{align*}
\Phi_t(xQ)& = (1-t)xQ+txQ[(xQ)^\intercal xQ]^{-1/2}\\& = (1-t)xQ+tx(x^\intercal x)^{-1/2}Q\\& = \Phi_t(x)Q.\end{align*}\end{proof}
The matrix $x^\intercal x$ is referred to as the Gram matrix of the columns of $x$, and the orthogonalization is known by the name L\"owdin orthogonalization (see \cite{Low}).

It is noted for the interested reader that this linear deformation retraction realizes the shortest path from $GL(n)$ to $O(n)$ in the Frobenius norm (i.e., the Euclidean norm) and in fact the cut locus for $O(n)$ in $\mathbb R^{n\times n}$ is exactly $\det^{-1}(0)$, the singular matrices.

\begin{thm}\label{regularization}
The space of unlabeled simplices in $\mathbb R^n$ deformation retracts to the space of regular simplices. 
\end{thm}

\begin{proof} Let $A$ be the $n\times n$ symmetric matrix whose columns form a unit edge length simplex.  Explicitly $A$ has $\mu$ for each entry on its diagonal and $\nu$ for each entry off the diagonal where $\mu^2+(n-1)\nu^2=1$ and $2(\mu-\nu)^2=1$, so that
\[\mu=\frac{n+\sqrt{n+1}-1}{\sqrt 2 n}\hspace{1.5cm}\text{and}\hspace{1.5cm}\nu=\frac{\sqrt{n+1}-1}{\sqrt 2 n}.\]

Then up to scaling and translation, the space of regular labeled simplices in $\mathbb R^n$ is the orbit $O(n)\cdot A$.

Let $B_i$ be the $n\times n$ identity matrix with the $i$th row replaced by $[-1,\ldots,-1]$.  The matrix $B_i$ acts on the right as a column operator to change bases between vertices of an $n$-simplex.  I.e., given a matrix $x$ whose columns $x_i$ form a basis, the columns of $xB_i$ are those emanating from $x_i$ to $0$ and to each of the other $x_j$'s (see figure \ref{figu7o5}).  The matrices $B_i$ generate a representation of $S_{n+1}\in O(n)$ with $B_i$ mapped to by the transposition $(0,i)$ (see figure \ref{figu7o7}).  Let $B\in\{B_i\}$.  Then we have the relationship
\[AB=QA, \hspace{1cm} {\rm so }\hspace{1cm}BA^{-1}=A^{-1}Q\]
for some $Q\in O(n)$ (this is obvious, geometrically).  
\begin{figure}[h] \centering \def\svgwidth{150pt} 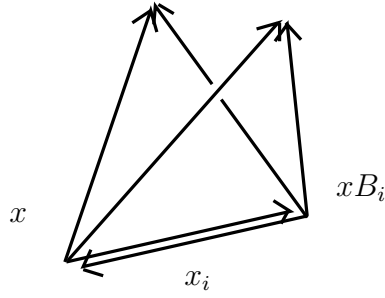 \caption{$B_i$ swaps $x$ for the basis at $x_i$ which spans the same simplex as $x$.}
\label{figu7o5}
\end{figure}

\begin{figure}[h!]
\centering
\includegraphics[scale=.5]{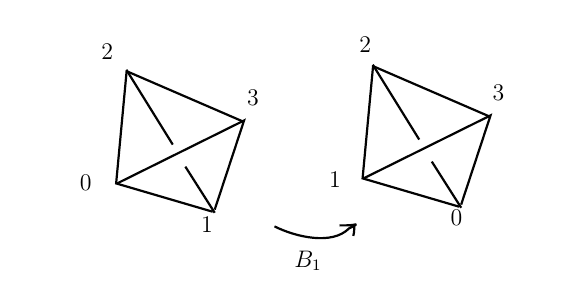}
\caption{$B_i$ is effectively the transposition $(0,i)$.}
\label{figu7o7}
\end{figure}

\begin{figure}[h] \centering \def\svgwidth{320pt} 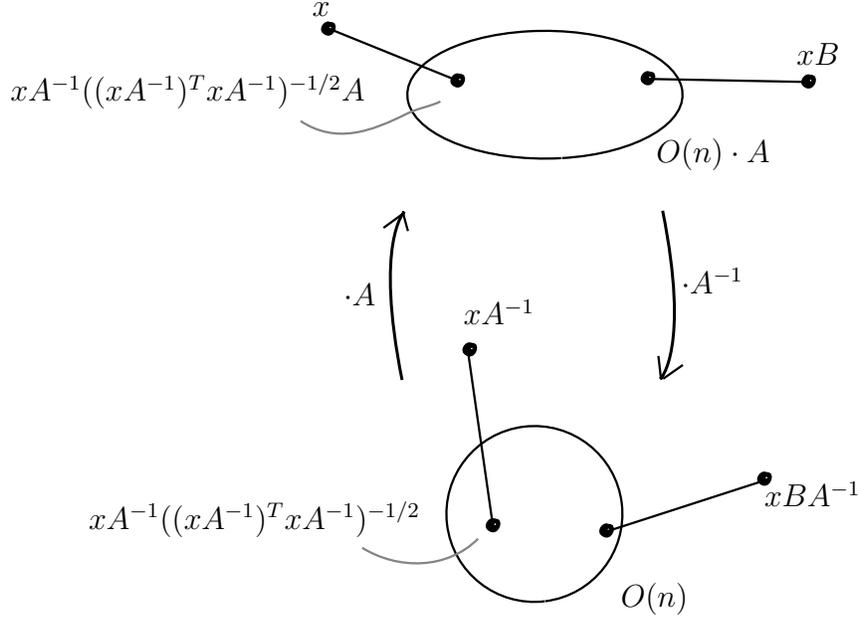 \caption{The $O(n)$-equivariant orthogonalization is conjugated to give an $S_{n+1}$-equivariant regularization.  The line segments are the deformation retractions in $GL(n)$.}
\label{figu7o6}
\end{figure}

Let \begin{align*}\Omega_t(x) &=\Phi_t(xA^{-1})A \\&=\Big[(1-t)xA^{-1}+txA^{-1}((xA^{-1})^\intercal(xA^{-1}))^{-1/2}\Big]A.\end{align*}
Then $\Omega_t(x)\cdot A^{-1}$ gives a linear path from $x$ to $O(n)\cdot A$ (see figure \ref{figu7o6}).  Equivariance of $\Omega$ in $B$ follows from equivariance of $\Phi$ in $O(n)$: \begin{align*}\Omega_t(xB) &=\Phi_t(xBA^{-1})A \\ &= \Phi_t(xA^{-1}Q)A \\ &= \Phi_t(xA^{-1})QA \\ &= \Phi_t(xA^{-1})AB\\&=\Omega_t(x)B.\end{align*}
It is therefore the case that $\Omega_t$ descends to the quotient $(GL(n)\cdot A)/S_{n+1}$, to give a linear (i.e., vertices move along linear paths) $S_{n+1}$-equivariant regularization of simplices in $\mathbb R^n$.
\end{proof}

\begin{thm} \label{P(K)} The space $C(K)$ of unlabeled codimension 2 skeleta of the $n$-simplex linearly embedded in $\mathbb R^n$ deformation retracts to the pyramidal space $P(K)$.
\end{thm}

\begin{proof} Over $\alpha^{-1}(0,\frac 1 2 V]$ we use $\Omega_t$ to regularize simplices.  For the rest, we use $\alpha$ as a parameter to alter $\Omega$ in two ways.  First, as $\alpha$ nears $V$ we wish to leave the {\em wide face}, (i.e., the face opposite the large solid angle vertex) ever more fixed.  Second, we wish to use $\alpha$ as a parameter to scale the terminal simplex so that it is not the height of a regular simplex but rather as $\alpha$ approaches $V$, the height of the terminal simplex approaches 0. Let $\eta(x)=2\alpha(x)/V-1$ be a reparametrization of $\alpha$ on $\alpha^{-1}(\frac 1 2 V,V)$ (so $\eta$ ranges from 0 to 1 as $\alpha$ ranges from $\frac 1 2 V$ to $V$).  

For any path $\Gamma(t)$ originating in $\alpha^{-1}(\frac 1 2 V,V)$ there is a unique orientation-preserving affine linear transformation $\gamma_t(x)\in {\rm Aff}_+(\mathbb R^n)$ which agrees with $\Gamma$ on the wide face, and which scales isometrically in the perpendicular direction (note that $\gamma_0={\rm id}$).  We say $\gamma$ is the {\em map induced by} $\Gamma$.  Let $\omega_t$ be thus induced by $\Omega_t$.  We consider $\omega_{\eta(x) t}^{-1}\circ\Omega_t(x)$.

Note, for ease of understanding where this argument is going, that $\omega_{t}^{-1}\circ\Omega_t(x)$ is a path from $x$, which keeps the wide face $W_x$ fixed and ends in being height $\sqrt{\frac {n+1}{2n}}$ (i.e., the height of a regular $n$-simplex with unit edges) above $W_x$, directly over its barycenter, while $\omega_{0}^{-1}\circ\Omega_t(x)=\Omega_t(x)$. Also, for a path $\Gamma(t)$ originating in $\alpha^{-1}(\frac 1 2 V,V)$, there is a map $\bar\Gamma_t(x)\in {\rm Aff}_+(\mathbb R^n)$ which fixes the plane containing the image of $W_x$ under $\Gamma(t)$ and scales by $(1-\eta(x))t$ in the perpendicular direction.  Then
\begin{align}\Psi_t(x)=\bar\Omega_t\circ\omega_{\eta(x) t}^{-1}\circ\Omega_t(x)\end{align}
is the linear deformation retraction that results in simplices that differ from being pyramids by an affine linear map which regularizes the wide face and is extended to an isometry in the perpendicular direction. (Technically we should also translate so the barycenter is at the origin, although this is immaterial.)  For this final step, we use the deformation retraction of theorem \ref{regularization} in dimension $n-1$ to regularize the wide face, as it is a non-degenerate $(n-1)$-simplex.

This process has a unique continuous extension to $\alpha^{-1}(V)$, which is that the non-extremal vertex moves along a straight line to the barycenter of its wide face, then the wide face is regularized.

To recap: we regularize those simplices with a small greatest solid angle.  For those with a large enough greatest solid angle to designate a vertex, and hence its opposite face, we use this solid angle as a parameter to damp the effect of the regularization on the wide face, and simultaneously to scale the resulting simplex to be pyramid like (except that the wide face is not yet regular).  We follow this with a regularization of the wide face, which is an isometry in the perpendicular direction.  It is noted that, as it is, the space $P(K)$ flows in the direction toward the regular subspace, so the deformation retraction is {\em weak} in the sense that the target space moves.  It should be clear that a reparametrization can fix this, if such is called for.
\end{proof}

To understand the degeneracies of $C(L)$, again we apply Radon's theorem (Theorem \ref{Radons}).  In this case it tells us either some vertex of $x\in C(L)$ is in the interior of the convex hull of the other vertices, or that there is an edge which intersects its opposite would-be $(n-1)$-face (this face is not part of $x$).  Intersections of faces, each of dimension greater than 1, is forbidden, as both faces will belong to $x$.  Then there are essentially two types of generic configurations in $C(L)$, connected by those $x$ with one vertex in a would-be $(n-1)$-face, which are present in $P(L)$ (see Figure \ref{figu7o3} and the right hand column of Figure \ref{7.12}).

\begin{figure}[h] \centering \def\svgwidth{240pt} 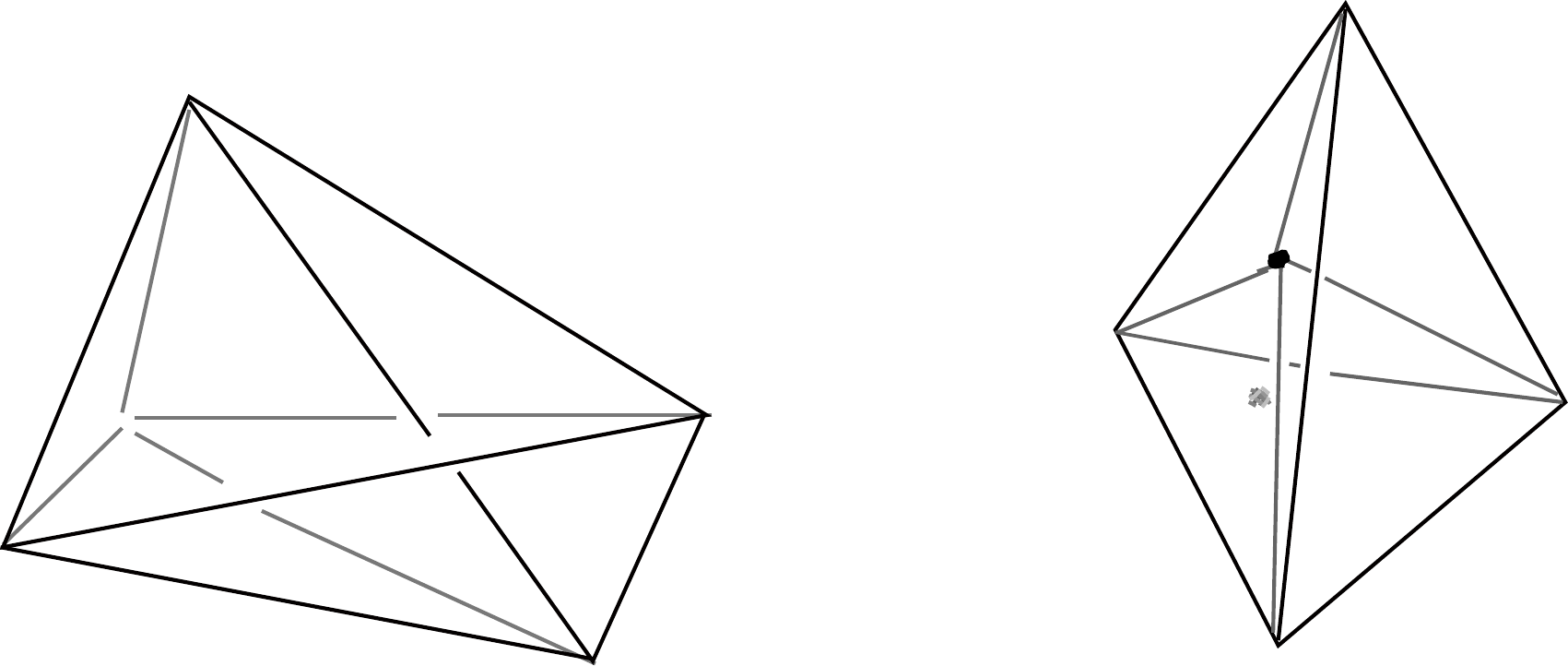 \caption{Generic configurations of $C(L)$: either some vertex is interior to the convex hull of the others or some specific edge intersects its opposite face.}
\label{figu7o3}
\end{figure}

\begin{thm} \label{P(L)} The space $C(L)$ of unlabeled codimension 3 skeleta of the $(n+1)$-simplex, linearly embedded in $\mathbb R^n$ deformation retracts to the pyramidal space $P(L)$. \end{thm}

\begin{proof} We achieve  the deformation retraction in three steps, the first two of which are divided into 3 cases each.  By $\mathcal I\subset C(L)$ we name those configurations with a vertex interior to the convex hull of the others.  By $\mathcal E$ we name those with an edge intersecting the would-be $(n-1)$-face in the interior of that edge.  By $\mathcal B$ we denote their mutual boundary (see the middle right figure in Figure \ref{7.12}).

\begin{figure}[h!]
\centering
\includegraphics[scale=.5]{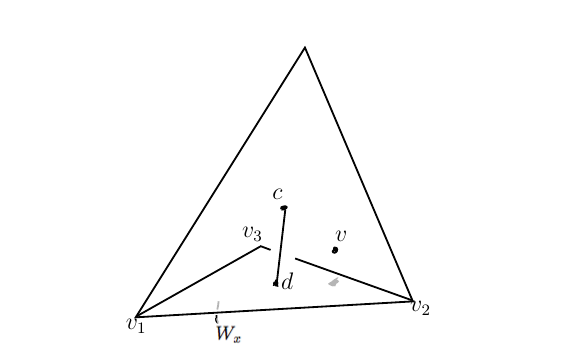}
\caption{Realize $v$ as a convex combination of $c$ and the closest vertices $v_i$ to $v$.}
\label{figu7o8}
\end{figure}

Step 1a. Let $x\in\mathcal I$ with interior vertex $v$.  Let $\{v_i\}$, $1\le i\le n$, be $n$ of the closest vertices of $x$ to $v$, let $c$ be the centroid of $x$, and $d$ be the centroid of the face $W_x$ spanned by $\{v_i\}$ (see figure \ref{figu7o8}).  We will move $v$ to lie along the line segment connecting $c$ to $d$.  This can be done explicitly by putting $v$ in barycentric coordinates 
\[v=qc+\sum a_iv_i.\hspace{1.5cm}(\text{with }q+\sum a_i=1\text{, and }q,a_i\ge 0)\]
Set $m=\min\{a_i\}$.  Note that $m=q=0$ cannot happen since this would put $v$ in the $(n-2)$-skeleton of $x$ (see figure \ref{figu7o9what}).  We have $3m\le 1-q$ and require a parameter $s(m,q)$ so as to send $v$ to $(1-s)d+sc$ which is continuous on $0\le 3m\le 1-q\le 1$ minus the origin $q=m=0$, and for which $s(0,q)=1$, $s(m,0)=0$ and $s(\frac 1 3 (1-q),q)=q$ (so that if $v$ is equidistant to two extremal vertices it gets sent to $c$, if it is in $W_x$ it gets sent to $d$, and if it is on the line connecting $c$ to $d$ it is fixed).  This is accomplished with
\[s(m,q)=(1-q)\Big(1-\frac{3m}{1-q}\Big)^{1/q}+q,\]
which we extend continuously by $s\equiv 0$ on $q=0$ (see figure \ref{figu7o10what}).  Sending $v$ to $(1-s)d+sc$ along the straight line path $v_t=(1-t)v+t(1-s)d+sc$ gives a retraction of $\mathcal I$ to the subspace of $\mathcal I$ with internal vertex along a radial segment connecting the barycenter to the center of a face. (Note that the would-be faces which include $c$ but exclude two of extremal vertices are the boundaries defining which radial segment $v$ ends up on.  Any vertex in this would-be face ends up at $c$.)

\begin{figure}[h!]
\centering
\includegraphics[scale=.5]{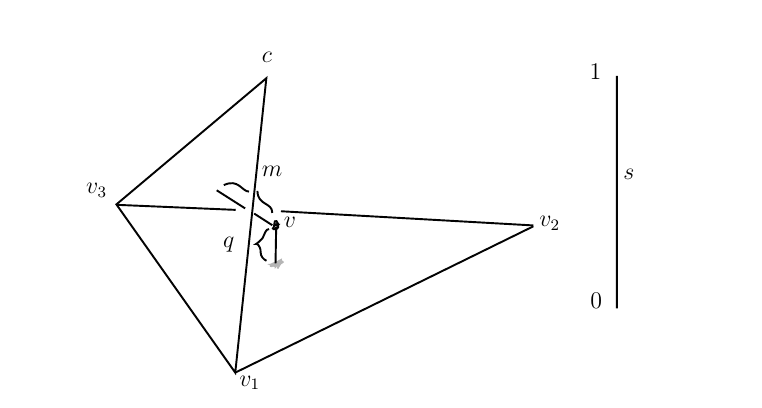}
\caption{Using the parameters $q$, which is distance from the extremal face, and $m$, minimum distance to a face containing $c$, to define $s$.}
\label{figu7o9what}
\end{figure}

\begin{figure}[h!]
\centering
\includegraphics[scale=.46]{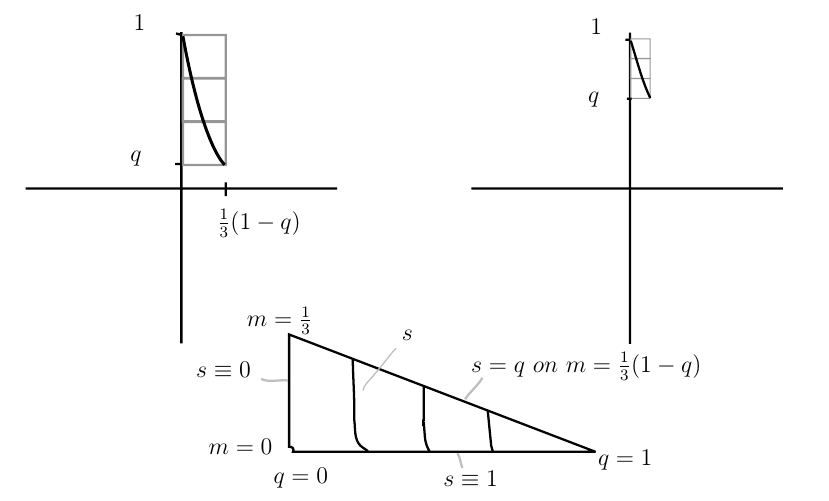}
\caption{Graphs of $s$ for smaller $q$ (left), for larger $q$ (right), and in the $q$-$m$ plane (bottom).  Note the origin is excluded. $~~~~~~~~~~~~~~~~~~~$$~~~~~~~~~~~~~~~~~~~~~~~~~~~~~$$~~~~~~~~~~~~~~~~~~~$$~~~~~~~~~~~~~~~~~~~~~~~~~~~~~$}
\label{figu7o10what}
\end{figure}

\begin{figure}[h!] \centering \def\svgwidth{350pt} 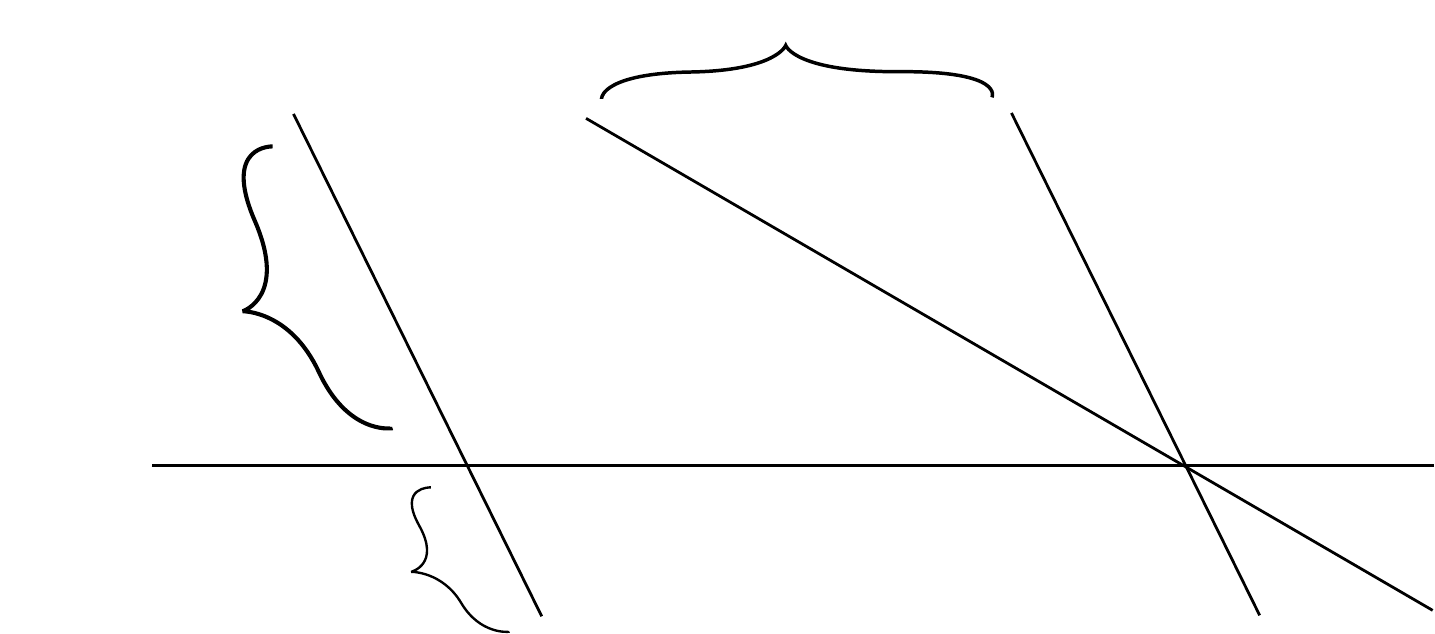 \caption{Parallel transport followed by a shear, with $v_2$ going back in the direction parallel transported.}
\label{fig7o11}
\end{figure}

Step 1b.  Let $x\in\mathcal E$. We want to parallel transport the edge $e$ containing the Radon point $p$ so that the intersection of this edge with its opposite face $W_x$ is at the barycenter $d$ of that face.  When one vertex $v_1$ of $e$ is close to $W_x$ we need the other vertex $v_2$ to move only a small distance so that step 1b can be continuously glued to step 1a.  To do this, we follow the parallel transport with a sheer back in the direction that $v_2$ has moved, in the plane containing $d$ and $e$, with origin at $d$, in proportion to $1-\frac{|v_1-p|}{|v_2-p|}$.  Explicitly, put $\ell_i=|v_i-p|$ where $\ell_2\ge \ell_1$  (see figure \ref{fig7o11}), and send $v_2$ along a linear path to $\bar v_2=v_2+\frac{\ell_1}{\ell_2}(d-p)$, and send $v_1$ along a linear path to $\bar v_1=v_1+(1+\frac{\ell_1}{\ell_2}(1-\frac {\ell_1}{\ell_2}))(d-p)$.  Then the weighted average $\frac {\ell_2} {\ell_1+\ell_2} v_1 + \frac {\ell_1} {\ell_1+\ell_2} v_2 = p$, whereas $\frac {\ell_2} {\ell_1+\ell_2} \bar v_1 + \frac {\ell_1} {\ell_1+\ell_2} \bar v_2 = d$.  As we approach $\mathcal B$, $\ell_1/\ell_2$ approaches $0$ and $v_2$ moves less and less.  Note also that as we approach $\mathcal B$, $v_1$'s path approaches the linear path to $d$.

Step 1c. These two deformation retractions agree on their respective extensions to $\mathcal B$.  In both cases the extension is to send the Radon point vertex to the centroid of the $(n-1)$-face it is in, in a linear path while fixing everything else.

At the end of step 1 the Radon point of each $x\in C(L)$ is along a ray extending from the centroid of $x$ to the centroid of a face.  For $0\le t \le \frac 1 3$, let $\Lambda_t$ be all three parts of step 1, simultaneously performed in the variable $3t$.

Step 2a. For $x\in\Lambda_{1/3}(\mathcal I)$ (we now have $q=s$), we use the parameter $1-q$ for the role of $\eta$ in equation (1) (i.e., in the definition of $\Psi$) to damp the regularization of the extremal $n$-simplex against $W_x$.  We do not scale in the perpendicular direction as was done via $\bar \Omega$.  If we write 
\begin{align}\hat\Psi=\bar\Omega_t\circ\omega_{(1-q(x)) t}^{-1}\circ\Omega_t(x)
\end{align}
where $\bar\Omega_t$ scales by $qt$ in the direction perpendicular to $\Omega_t(W_x)$, then step 2a. can be succinctly written as $\bar\Omega_t^{-1}\hat\Psi_t$.  After this step, those configurations which had a vertex already at their barycenter (so that $1-q=0$) have been brought into $P(L)\cap\mathcal I$ (specifically, those pyramidal configurations with $A_{n+1}$ stabilizers in $SO(n)$, i.e., they are regular with a vertex at their barycenter), while those with $q\in(0,1)$ arrive at a configuration which only fails to be in $P(L)$ by exactly the map which regularizes $W_x$ and extends preserving distance and orientation in the perpendicular direction.

Step 2b.  For $x\in\Lambda_{1/3}(\mathcal E)$ let $v_1,v_2$ be as in 1b.  We will keep the shared face fixed, by operating as if $q=0$ (here $q$ no longer represents a barycentric coordinate, but merely the parameter that replaces $1-\eta$ in equation (1)) and applying $\bar\Omega_t^{-1}\hat\Psi_t$ to the half-space $H_2$ containing $v_2$, while to the other half-space $H_1$ containing $v_1$ we apply $\hat\Psi_t$.  This step also results in a configuration which only fails to be in $P(L)$ by exactly the map which regularizes $W_x$ and extends preserving distance and orientation in the perpendicular direction.

Step 2c.  For $x\in\Lambda_{1/3}(\mathcal B)$ the limits of two processes 2a. and 2b. agree.

For $\frac 1 3 \le t\le \frac 2 3$, let $\Lambda_t$ be all three parts of step 2, simultaneously performed in the variable $3t-1$.

At the end of step 2 all that remains is to regularize $W_x$ in the hyperplane it spans, and extend to an orientation preserving isometry on $W_x^\perp$.  This is achieved by theorem \ref{regularization} with the time parameter $3t-2$.  This gives the final third of $\Lambda$.

We have given $\Lambda_t:C(L)\rightarrow P(L)$, a deformation retraction from the configuration space of the $(n-2)$-skeleton of the $(n+1)$-simplex linearly embedded in $\mathbb R^n$, thus proving Theorem \ref{P(L)}.  As noted before, this is a weak deformation retraction, as is, but can easily be reparametrized to be a strong deformation retraction.

\end{proof}

Theorem \ref{P(K)} and Theorem \ref{P(L)} along with Proposition \ref{cylinderprop} comprise the main result:

\begin{thm}\label{maintheorem}For $n>2$, $C(X)$ (for either $X=K,L$) has the homotopy type of the double mapping cylinder \[SO(n)/A_{n+1}\leftarrow SO(n)/A_n\rightarrow SO(n)/S_n,\] where $A_n$ is the alternating group and $S_n$ the symmetric group.\end{thm}

The spaces $C(L),C(K)$ are covered by their respective labeled analogs $E(X)=Emb_{Linear}(X,\mathbb R^n)$, the space of linear embeddings of $X=C,L$ in $\mathbb R^n$.   The deformation retractions above lift to these covers, so that $E(X)$ deformation retracts to its subspace of labeled pyramids.  In the case $X=K$, $n>2$, the space of labeled pyramids is easily seen to be homeomorphic to $SO(n)\times S(n)$ where $S(n)$ is the graph resulting from the suspension of $n+1$ points (so $S(n)\simeq \bigvee^nS^1$, a wedge of $n$ circles).  The interiors of the edges parametrize the non-regular pyramids, so that the degenerate pyramids are at the midpoints and half-high pyramids are $1/4$ from either end, depending on the orientation of the labeling.  

In the case $X=L$, $n>2$, the labeled pyramids with full $A_n$ symmetry (i.e., a configuration made up of a regular simplex along with a vertex at its center) are partitioned into $n+1$ components, each one corresponding to the vertex at the barycenter of the others.  Any two such components are connected in $P(L)$ by a cylinder $SO(n)\times I$, the second component of which parametrizes the central vertex leaving through a face followed by the vertex oppose this face entering the simplex.  The homotopy type of $E(L)$ is therefore $SO(n)\times K_{n+1}\simeq SO(n)\times S(n)$.  We summarize these results as a corollary, as they follow from the deformation retractions of the respective configuration spaces.

\begin{cor} For $n>2$, $X=K,L$, the space of embeddings $E(X)$ has the homotopy type of $SO(n)\times S(n)$.
\end{cor}

%
%
%
% REGULARIZING GEOMETRICALLY
%
%
%

\section{Regularizing Geometrically}

The ideas contained in this section were initial attempts at the regularizations in the deformation retractions of Section \ref{section 1}.  They are included only for their geometric appeal.  Nothing in this section strengthens the results of Section \ref{section 1}.

The case $n=3$ is special because there is a direct way to produce a regularization of tetrahedra from an equivariant orthogonalization of basis.  To each tetrahedron we assign what we will call its {\em bimedian basis} which is the (unordered) collection of 3 line segments joining midpoints of opposite (necessarily skew) edges (see figure \ref{bimedianbasis}).  These line segments intersect at the barycenter, which bisects each line segment.  It is easy to verify that $E$ the {\em standard bimedian basis}--i.e., the basis formed by the standard basis vectors and their negations--has exactly 2 tetrahedra which have $E$ for a bimedian basis, which differ by the reflection $-I$.  Any other bimedian basis is then the image of this one under an invertible linear map (modulo translation), so that the space of tetrahedra is a double cover of the space of bimedian bases.   Any deformation retraction of $GL(3)$ to $O(3)$ which is signed permutation equivariant (i.e., equivariant in $\mathbb Z_2\wr S_3<O(n)$) descends to a deformation retraction of bimedian bases to the {\em orthonormal bimedian bases}, which then lifts to the double cover, resulting in regularizing of the tetrahedra.  Thus the L\"owdin process gives a regularization of tetrahedra in $\mathbb R^3$.

This process does not generalize to arbitrary $n$ in any obvious way.  It relies on a homomorphism from $S_{n+1}$, the symmetries of the $n$-simplex, to $\mathbb Z_2\wr S_n$, the symmetries of the bimedian basis.  The image of this homomorphism must at least generate $S_n<\mathbb Z_2\wr S_n$, and so must be injective for $n>3$, since $S_{n+1}$'s only normal subgoup is $A_{n+1}$.  The respective orders are $(n+1)!$ and $2^nn!$, thus such a method can only exist when $n=2^k-1$ for some $k$.

\begin{figure}[h] \centering \def\svgwidth{150pt} 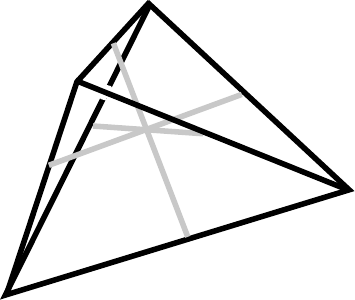 \caption{The bimedian basis of a tetrahedron is shown in gray.}
\label{bimedianbasis}
\end{figure}

For general $n$, one might regularize a simplex by inflating its insphere while fixing the volume of the simplex.  We show here that indeed this works.

\begin{lemma} For $r_x$ the inradius of an $n$-simplex $x$ in $\mathbb R^n$ we have
\[r_x=n \cdot Vol(x)/Vol(\partial x)\]
\end{lemma}

\begin{proof}
Realize $x$ as a cone over $\partial x$ to the incenter.  Partition this cone into the cones over each face $f_i$.  The volume of the cone over $f_i$ is $\frac 1 n\cdot r_x\cdot Vol(f_i)$.  Summing over the faces gives the result.  (See figure \ref{coneoverface}).
\end{proof}

\begin{figure}[h] \centering \def\svgwidth{150pt} 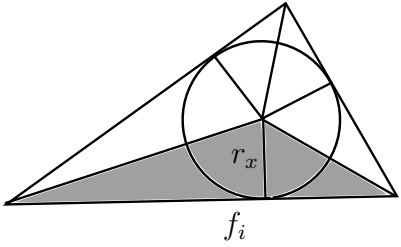 \caption{The volume of the simplex is disassembled into simplices with height $r_x$ above base face $f_i$.}
\label{coneoverface}
\end{figure}

By the above, flowing along the gradient of $Vol(\text{insphere}(x))$ constrained to a fixed simplex volume is the same as flowing to minimize the surface volume, with the same constraint.  We consider the component of this flow in the direction which fixes a base face $f_v$ and moves its opposite vertex $v$ at height $H$ above $f_v$, to minimize $Vol(\partial x_t)$ to prove the following.

\begin{lemma} The flow which minimizes the surface volume of a simplex $x$, subject to maintaining a fixed volume, results in a simplex where each vertex is directly over the incenter of its oppose face.
\end{lemma}

\begin{proof}
Let the $n$ $(n-2)$-dimensional faces of $f_v$ be indexed as $g_i$, and denote the $(n-1)$-dimensional face containing $g_i$ and $v$ with $\bar g_i$.  Let $a_i$ be the signed distance from the projection of $v$ on the hyperplane containing $f_v$ to $g_i$, signed so that $a_i$ is positive whenever the projection of $v$ is in $f_v$ (see figure \ref{triangleproj}).  We have \[Vol(\bar g_i)=\frac 1 {(n-1)}\cdot Vol(g_i)\cdot \sqrt{H^2+a_i^2},\] so that
\[Vol(\partial x)=Vol(f_v)+\frac 1 {(n-1)}\sum Vol(g_i)\sqrt{H^2+a_i^2}.\]

\begin{figure}[h] \centering \def\svgwidth{150pt} 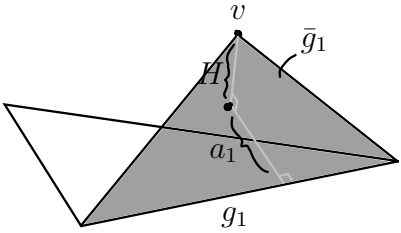 \caption{This figures illustrates $a_i$, $g_i$ and $\bar g_i$ for the $n=3$ case.}
\label{triangleproj}
\end{figure}

Any $a_i$ depends affine-linearly on the others, since removing any one gives a coordinate system, so that 
\begin{equation}1=\sum C_ia_i\end{equation} for some constants $C_i$.  The value of $a_i$ for $v$ over the $i$th vertex, for which all other $a_j$'s are 0, is the altitude $A_i$ of that vertex in $f_v$, giving $C_i=\frac 1 {A_i}$ and
\begin{equation} A_i Vol(g_i)=(n-1) Vol(f_v).
\end{equation}
Then (3) gives the constraint $\sum \frac {a_i} {A_i}=1$ and using the method of Lagrange multipliers we get the system 
\[\sum \frac {a_i} {A_i}=1\]
and
\[\frac 1 {(n-1)}\frac {Vol(g_i)\cdot a_i}{\sqrt{H^2+a_i^2}}=\frac \lambda {A_i},\]
which using (4) simplifies to
\[Vol(f_v)\cdot \frac {a_i} {\sqrt{ H^2+a_i^2}}=\lambda\]
which gives
\[\frac {a_i} {\sqrt{H^2+a_i^2}}=\frac {a_j}{\sqrt{H^2+a_j^2}}\]
implying
\[a_i^2(H^2+a_j^2)=a_j^2(H^2+a_i^2)\]
which necessitates $a_i=a_j$ since both are positive where a minimum is achieved.

It is therefore the case that volume of the boundary is minimized when the vertices are directly over the incenters of their respective opposite faces.

It remains to argue that such a trajectory actual terminates in a simplex with the property that the vertices are directly over the incenters of their opposite faces, as opposed to escaping to ``infinity'' or limiting to more than a single point.

Note first that if we have vertices of arbitrary distance $d$ from the incenter, then the cone formed by the vertex and the insphere (i.e., truncate it where its boundary intersects the insphere) is contained in the simplex $x_t$, and has volume with $\liminf$ equal to that of $d\cdot c\cdot (1/n)$ where $c$ is the volume of the $(n-1)$-ball spanned by a great sphere of the insphere.  Then that the volume of $x_t$ is fixed and is an upper bound for this cone, necessitates that the inradius vanishes, contradicting the construction of the flow.  Also note that if $\ell$ is the altitude of $v$ and $w$ is the closest vertex of $x$ to $v$ with edge length $|(v,w)|$, then for $2r$ the indiameter, we have $2r\le \ell \le |(v,w)|$, so that again $r$ must vanish, contradicting the construction of the flow. (Figure \ref{escapetoinfty} illustrates these two arguments).  Translation to infinity is clearly not a concern.  For example we can further stipulate that the incenter is fixed at the origin.

%Finally, getting arbitrarily close to the critical set gives that each vertex gets arbitrarily close to being directly over the incenter of the opposite face, so that the flow results in a single limiting simplex.\\

\begin{figure}[h] \centering \def\svgwidth{250pt} 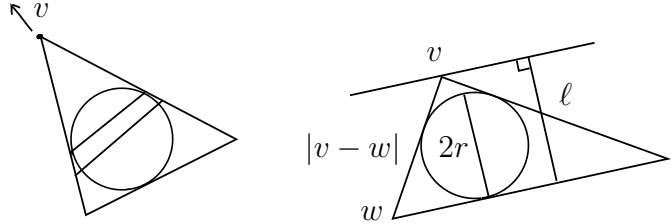 \caption{The trajectory does not escape to infinity.}
\label{escapetoinfty}
\end{figure}

 \end{proof}
 
 \begin{lemma} A simplex for which each vertex orthogonally projects to the incenter of the opposite face is a regular simplex.
\end{lemma}
\begin{proof}
Let $v,w$ be vertices of the simplex $x$, $c_v$ be the incenter of the face $f_v$ opposite $v$ and $r_w$ be the outward pointing radial vector from $c_v$ to the codimension 2 face excluding $w$ and $v$ (Figure \ref{incentersym} is helpful).  Note that the $r_w$ form congruent right triangles with $v-c_v$, and that on the $i$th face the gradient of the distance to $f_v$, (at $c_v+r_w$ in $f_w$) is the hypotenuse of the right triangle containing $r_w$.  Thus the point on $v-c_v$ which is equidistant to some face and to $c_v$ is actually the incenter $c_x$.

\begin{figure}[h] \centering \def\svgwidth{150pt} 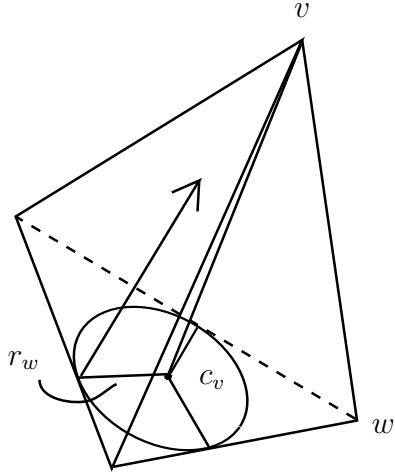 \caption{The condition that each vertex is over its opposite incenter implies regularity.}
\label{incentersym}
\end{figure}

It is therefore the case that $c_x$ projects orthogonally to $c_v$ and all other faces have equal pitch relative to $f_v$.  That is, for $c_w$ the outward pointing vector from $c_x$ to $f_w$ realizing the inradius, we have that $c_w\cdot c_u=c_w\cdot c_y$ for all distinct $u,w,y$.  It follows that the simplex they define has full symmetry.
\end{proof}

That a single regular simplex is the limit of any trajectory follows from the fact that the flow is similarity-equivariant.  Specifically, the orbit of a simplex under similarity transformations is contained in a level set of the {\em irregularity potential} function $Vol(x)/Vol(\rm{insphere}(x))$, so it cannot be the case that two distinct regular simplices, which differ by a similarity transformation, are limit points for some trajectory.

The previous two lemmas piece together to give the following theorem.
\begin{thm}\label{simplexTheorem}
The deformation retraction of Theorem \ref{regularization} is achieved by the flow which increases the inradius of the simplex while keeping its volume fixed.\\
\end{thm}

\begin{figure}[h!] \centering \def\svgwidth{180pt} 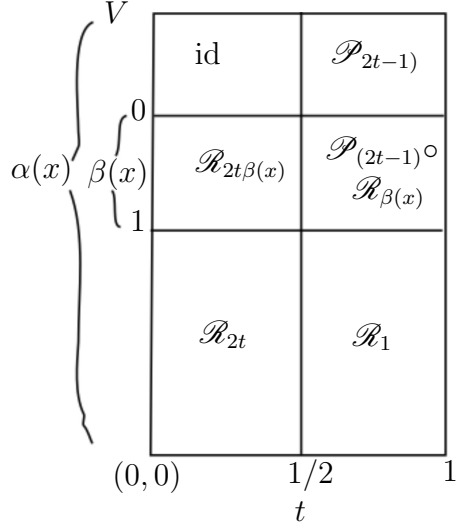 \caption{The schematic for gluing the regularization and preferred apex deformation retractions together.}
\label{schematic}
\end{figure}

\begin{prop} Any regularization deformation retraction of simplices can be extended to give a  deformation retraction from $C(K)$ to the pyramidal space $P(K)$.
\end{prop}

\begin{proof}
The parameter $\alpha$ resulting from Lemma \ref{solidangles} gives a preferred apex/face pair for configurations in $\alpha^{-1}(\frac 1 2 V,V]$.  These can be made into pyramids by sending the apex, along a straight line path parallel to its opposite face $F$, to be directly over the barycenter of $F$, followed by regularizing the wide face (via Theorem \ref{simplexTheorem} or \ref{regularization}).  Call this preferred apex deformation retraction $\mathscr P$ and call the regularizing deformation retraction $\mathscr R$.  Let 
\[\beta= \frac{3-4\alpha(x)}V\] be a reparametrization of $\alpha$ on $\alpha^{-1}(-\frac 1 2 V,\frac 3 4 V]$. Then $\mathscr R$ and $\mathscr P$ can be glued together by the schematic in Figure \ref{schematic}.  The entries in the diagram have been chosen so that the 6 regions glue together continuously.  \end{proof}

The idea is to use $\beta$ as a parameter with which to perform some of the regularization, followed by the preferred apex deformation retraction.  It is noted that, as in Section \ref{section 1}, pyramids for which $\beta(x)>0$ will move in the direction of being regular, so the deformation retraction is {\em weak} in the sense that the target space moves.  As before, it should be clear that a reparametrization can fix this, if such is called for.

%
%
%
%  PRESENTATIONS
%
%
%

\section{Presentation of $\pi_1(C(X))$ for $n=3$ and Action on $F_3$} 

In this section we give presentations for $\pi_1(C(X))$ in the case $n=3$, and show how this group acts on $F_3$.  The general $n$ case is similar to this specific case, although we do not present the general form here.

From Theorem \ref{maintheorem}, Van Kampen's theorem gives the fundamental group of $C(K)$ as $2 A_{n+1}\ast_{2 A_n} 2 S_n$, where the 2's represent the pull backs from the canonical quotient $q:{\rm spin}(n)\rightarrow SO(n)$.  For the case $n=3$, let $T=A_4$ be the orientation preserving isometries of the tetrahedron, and $\rm{Dic}_3$ be the {\em dicyclic} group $q^{-1}(D_3)$ for $D_3$ the symmetries of a triangle.  

\begin{cor}
For $n=3$, the fundamental group of $C(K)$ is 
\[2 T\ast_{\mathbb Z_6} \rm{Dic}_3\cong \langle X,R,S~|~X^2=R^3=S^3=(SR)^3,~XR=R^{-1}X\rangle\]
\end{cor}

Here $X$ is a rotation (which has order 4 in $\pi_1(C(K_4))$) which reflects a planar configuration via a rotation of $\pi$ in a given direction, and $R,S$ are two face rotations of the tetrahedron (up to conjugation by a path connecting the planar tetrahedra to the regular ones) as given in figure \ref{XRS}.\\

\begin{figure}
\centering
\includegraphics[scale=.75]{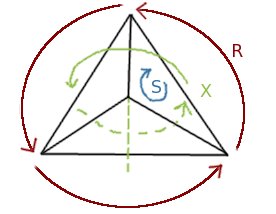}
\caption{The generators $X,R,S$.  We consider $R,S$ rigid motions of a regular configuration, while considering $X$ a rotation of $\pi$ of a planar configuration.}
\label{XRS}
\end{figure}

Another presentation of $\pi_1(C(K_4))$ is given in terms of loops from a base point in the planar configurations which transpose the center vertex and an extremal vertex by passing the center vertex up and over while passing the extremal vertex down and under.  (Figure \ref{generator1} shows such a generator).   This presentation has two advantages.  First, it is particularly simple and is symmetric, in the sense that $Aut(\pi_1(C(K_4)))$ acts transitively on it.  Second, it makes transparent the action of $\pi_1(C(K_4))$ on the free group on three generators $F_3$, the fundamental group of the complement of a given configuration.\\

\begin{figure}[h]
\centering
\includegraphics[scale=.6]{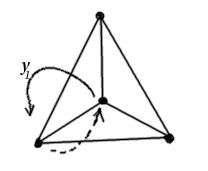}
\caption{The generator $y_1$ which transposes the center vertex with the one in position 1, by passing the center up and over while passing the extremal vertex down and under.}
\label{generator1}
\end{figure}

\begin{prop}
The fundamental group of $C(K_4)$ is generated by three elements $\{y_1,y_2,y_3\}$, subject to the following relations.  
\begin{align}
%number these with i's
(y_j y_i^{-1})^3&=(y_k y_l^{-1})^3{\rm ~for ~neither ~side ~trivial},\tag{i}\\
 y_i y_j^{-1} y_i&= y_j^{-1} y_i y_j^{-1} ~{\rm for~}i\ne j,\tag{ii}\\
y_k y_j^{-1}y_iy_j^{-1}y_k&=y_j^{-1}y_iy_j^{-1} {\rm ~for~} i,j,k {\rm~distinct}.\tag{iii}\end{align}\end{prop}

The isomorphism can be seen by observing one presentation in terms of the geometry of the other, the fine details of which we omit.  The map is given by the identities
\[\begin{array}{rclcrcl}
S&=&y_3^{-1}y_2&&y_3&=&X^{-1}SR^{-1}S^{-1}\\
R&=&y_2^{-1}y_3y_1^{-1}y_2&~~~~\rm{and}~~~~&y_2&=&RX^{-1}SR^{-1}S^{-1}R^{-1}\\
X&=&y_3^{-1}y_1y_3^{-1}&&y_1&=&R^{-1}X^{-1}SR^{-1}S^{-1}R\end{array}\]

As in the case of the braid group, there is a ``pure'' subgroup of $C(K_4)$ which returns vertices to their original position, which is precisely $\pi_1(E(K_4))=F_3\times \mathbb Z_2$.  This is the kernel of the map $\pi_1(C(K_4))\rightarrow S_4$.  In terms of the generators $\{y_i\}$, this kernel is generated by each of the three $y_i^2$, for the left factor, and $\tau=(y_iy_j^{-1})^3$ (any two distinct $i,j$), for the right factor.  Geometrically, this can be seen by viewing $y_iy_j^{-1}$ as a rotation of the tetrahedron by $2\pi/3$ so that it cubes to a rotation of $2\pi$, which explains the first set of relations (i).  The second set of relations (ii) can be rewritten, by multiplying both sides by the left side, to state that $y_i y_j^{-1} y_i$ squares to $\tau$.  Geometrically this is so, because $y_i y_j^{-1} y_i$ is effectively a rotation of $\pi$ about the edge $e_k$ which would get reversed by $y_k$ (see figure \ref{gens323}).  The third set of relations (iii) can then be rewritten to state that conjugation of $y_k$ by this particular square root of $\tau$ inverts $y_k$.  This is easily seen from the fact that the circle along which the end points of $e_k$ travel under the action of $y_k$ gets reversed in orientation by $y_i y_j^{-1} y_i$.  It should be remarked that $\tau$ is thus central, as it commutes with $y_k$, for any $k$.  Also, we note from (i) that $(y_iy_j^{-1})^3=(y_jy_i^{-1})^3$ so that $(y_iy_j^{-1})^6=\tau^2=1$.\\

It is worth noting that the families (i),(ii) and (iii) of relations above are independent in the sense that no two families generate the third.  Without relation (iii) the quotient by the subgroup generated by $\{y_i^2\}$, $i=1,2,3$, and $(y_iy_j)^3$ has the Cayley graph of figure $\ref{CayleyComb}$.  In particular it is not finite and so is not $\Sigma_4$, thus (iii) is independent.  Restricting to a subgroup generated by two generators $y_i,y_j$ renders (iii) inconsequential, and gives $(y_iy_j^{-1})^3=(y_jy_i^{-1})^3$ as the only consequence of (i), so that it's easy to see (by a change of basis $h=y_i, g=y_iy_j^{-1}$, say) that (ii) is independent.  In fact, by abelianizing this subgroup (i.e., by counting the exponents in a relator) we have relations $(6,-6)=0$ from (i) and $(1,1)=0$ from (ii), in $\mathbb Z^2$, thus (i) is also independent.\\

\begin{figure}[h]
\centering
\includegraphics[scale=.6]{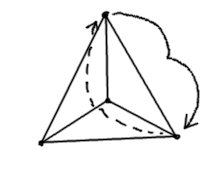}
\caption{The motion of $y_3y_2^{-1}y_3$ is effectively a rotation about the edge connecting the center vertex to the vertex in position 1.}
\label{gens323}
\end{figure}

\begin{figure}[h]
\centering
\includegraphics[scale=.6]{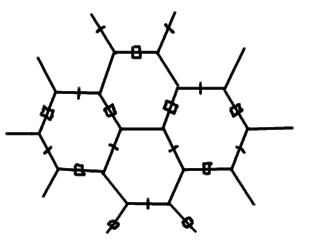}
\caption{The Cayley graph for the quotient which would otherwise result in $\Sigma_4$ after disposing of relation set iii.}
\label{CayleyComb}
\end{figure}

%        O
%        O
%       O O
%      O   O
%     OOOO
%    O       O
%   O         O
%  O           O CTION ON F_3                                             (6)

%\section{Action of $\pi_1(C(K_4))$ on $F_3$}\label{action}

\begin{figure}[h]
\centering
\includegraphics[scale=.7]{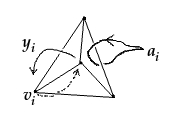}
\caption{The motion of a generator $y_i$, acting trivially on a loop $a_i$ in the complement of a configuration.}
\label{generatory_i}
\end{figure}

\begin{figure}[h]
\includegraphics[scale=.5]{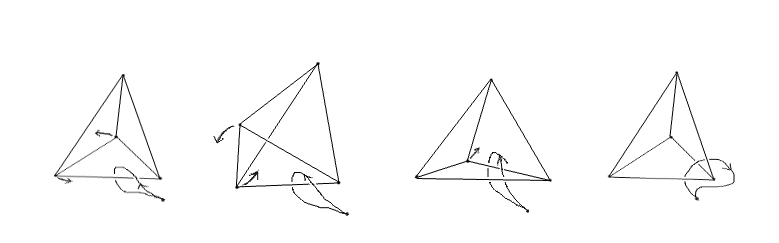}
\caption{The generator $y_1$ sends $a_2$ to $a_1a_2^{-1}$ (where concatenation of loops in $F_3$ is read from right to left).}
\label{gensaction}
\end{figure}

The complement of a linearly embedded tetrahedral graph in $\mathbb R^3$ has fundamental group $F_3$.  Unlike in the case of the braid group acting on the fundamental group of the complement of a configuration of points in the plane, here a rotation of the tetrahedron by $2\pi$ effects the trivial action on $F_3$.  That is, this loop, $\tau$, is in the kernel of the induced map $\psi:\pi_1(C(K_4))\rightarrow Aut(F_3)$.  By labeling the generators of $F_3$ in correspondence with the $y_i$'s of $\pi_1(C(K_4))$, (see figure \ref{generatory_i}), we have that 
\[\psi(y_i^2)(a)=a_iaa_i^{-1},\] for $a\in F_3$ and $a_i$ the generator of $F_3$ corresponding to $y_i$.  Thus $\psi\vert_{\pi_1(E(K_4))}$ is quotienting by the $\mathbb Z_2$ factor followed by the natural identification $F_3\cong Inn(F_3)$.  The action of $\pi_1(C(K_4))$ on $F_3$ is given by the identities 
\[y_i\cdot a_j=a_ia_j^{-1}\] if $i\ne j$ and otherwise 
\[y_i\cdot a_i=a_i,\] as seen in figures \ref{generatory_i}, \ref{gensaction}.  The generators of $\pi_1(C(K_4))$ are thus sent to square roots of conjugation in $Aut(F_3)$.

\begin{comment}

\end{comment}

\bibliography{LinearConfigurationsE2}

\begin{thebibliography}{10}

\bibitem{sbraids4}
Paolo Bellingeri.
\newblock On automorphisms of surface braid groups.
\newblock {\em J. Knot Theory Ramifications}, 17(1):1--11, 2008.

\bibitem{sbraids3}
Joan~S. Birman.
\newblock {\em B{RAID} {GROUPS} {AND} {THEIR} {RELATIONSHIP} {TO} {MAPPING}
  {CLASS} {GROUPS}}.
\newblock ProQuest LLC, Ann Arbor, MI, 1968.
\newblock Thesis (Ph.D.)--New York University.

\bibitem{wickets}
Tara~E. Brendle and Allen Hatcher.
\newblock Configuration spaces of rings and wickets.
\newblock {\em Comment. Math. Helv.}, 88(1):131--162, 2013.

\bibitem{Dahm}
David~Michael Dahm.
\newblock {\em A {GENERALIZATION} {OF} {BRAID} {THEORY}}.
\newblock ProQuest LLC, Ann Arbor, MI, 1962.
\newblock Thesis (Ph.D.)--Princeton University.

\bibitem{ongraphs1}
Robert Ghrist.
\newblock Configuration spaces and braid groups on graphs in robotics.
\newblock In {\em Knots, braids, and mapping class groups---papers dedicated to
  {J}oan {S}. {B}irman ({N}ew {Y}ork, 1998)}, volume~24 of {\em AMS/IP Stud.
  Adv. Math.}, pages 29--40. Amer. Math. Soc., Providence, RI, 2001.

\bibitem{Huh}
Youngsik Huh and Choon~Bae Jeon.
\newblock Knots and links in linear embeddings of {$K_6$}.
\newblock {\em J. Korean Math. Soc.}, 44(3):661--671, 2007.

\bibitem{invariant}
Riccardo Longoni and Paolo Salvatore.
\newblock Configuration spaces are not homotopy invariant.
\newblock {\em Topology}, 44(2):375--380, 2005.

\bibitem{Low}
Per-Olov L{\"o}wdin.
\newblock On the non-orthogonality problem connected with the use of atomic
  wave functions in the theory of molecules and crystals.
\newblock {\em J. Chem. Phys.}, 18:365--375, 1950.

\bibitem{sbraids1}
Luis Paris and Dale Rolfsen.
\newblock Geometric subgroups of mapping class groups.
\newblock {\em J. Reine Angew. Math.}, 521:47--83, 2000.

\bibitem{ongraphs2}
Lucas~Adam Sabalka.
\newblock {\em Braid groups on graphs}.
\newblock ProQuest LLC, Ann Arbor, MI, 2006.
\newblock Thesis (Ph.D.)--University of Illinois at Urbana-Champaign.

\bibitem{Ziegler}
G{\"u}nter~M. Ziegler.
\newblock {\em Lectures on polytopes}, volume 152 of {\em Graduate Texts in
  Mathematics}.
\newblock Springer-Verlag, New York, 1995.

\end{thebibliography}


\newpage
\begin{thebibliography}{9}


\bibitem[BH02]{BH02} T. Brendle and A. Hatcher, (2002) \emph{Configuration spaces of rings and wickets}, Commentarii Math. Helv. \textbf{88} (2013), 131-162.

\bibitem[D62]{D62} D. M. Dahm, (1962) \emph{A generalization of braid theory}, Ph.D. dissertation, Princeton
University, Princeton, N.J.

\bibitem[H02]{H02} A. Hatcher, (2002) \emph{Algebraic topology}, Cambridge University Press

\bibitem[HJ07]{HJ07} Y. Huh and C. B. Jeon, (2007) \emph{Knots and Links in Linear Embeddings of $K_6$}, J. Korean Math. Soc. \textbf{44} , No. 3, 661-671

\bibitem[L50]{L50} P.-O. L\"owdin (1950) J. Chem. Phys. \textbf{18}, 365

\bibitem[P07]{P07} L. Paris, \emph{Braid groups and Artin groups}, preprint,
arXiv:0711.2372 [mathGR]

\bibitem[Z95]{Z95} G. M. Ziegler, (1995), \emph{Lectures on Polytopes}, Graduate Texts in Mathematics \textbf{152}, Berlin, New York: Springer-Verlag.

\end{thebibliography}
\bibliographystyle{plain}

\end{document}